\documentclass{amsart}
\usepackage[english]{babel}
\usepackage{epsfig}
\usepackage{amssymb}
\usepackage[all]{xy}
\usepackage{enumerate}
\usepackage{mathbbol}

\newtheorem{theorem}{Theorem}[section]
\newtheorem{proposition}[theorem]{Proposition}
\newtheorem{corollary}[theorem]{Corollary}

\newtheorem{lemma}[theorem]{Lemma}

\newcommand{\Aut}{\mathord{\mathrm{Aut}}}

\newcommand{\Sol}{\mathord{\mathrm{Sol}}}

\newcommand{\pullbackcorner}[1][dr]{\save*!/#1-1.2pc/#1:(-1,1)@^{|-}\restore}

\newcommand{\RR}{\mathord{\mathbb{R}}}
\newcommand{\ZZ}{\mathord{\mathbb{Z}}}
\newcommand{\NN}{\mathord{\mathbb{N}}}

\newcommand{\N}{{\mathbb N}}
\newcommand{\Z}{{\mathbb Z}}

\newcommand{\Ker}[1]{{\operatorname{Ker}{#1}}}

\newcommand{\ignore}[1]{} 

\long\def\alert#1{\parindent2em\smallskip\hbox to\hsize
{\hskip\parindent\vrule%
\vbox{\advance\hsize-2\parindent\hrule\smallskip\parindent.4\parindent%
\narrower\noindent#1\smallskip\hrule}\vrule\hfill}\smallskip\parindent0pt}

\title{SOME ANOMALOUS EXAMPLES OF LIFTING SPACES}

\author[G. Conner]{Gregory R. Conner$^1$}
\address{ 
Department of Mathematics, 
Brigham Young University, 
Provo, UT 84602, USA} 
\email{conner@mathematics.byu.edu}

\author[W. Herfort]{Wolfgang Herfort}
\address{ 
Institute for Analysis and Scientific Computation  
Technische Universit\"at Wien 
Wiedner Hauptstra\ss e 8-10/101
Vienna, Austria} 
\email{wolfgang.herfort@tuwien.ac.at}

\author[P. Pavesic]{Petar Pave\v si\'c $^2$}
\address{ 
Faculty of Mathematics and Physics  
University of Ljubljana 
Jadranska 21
Ljubljana, Slovenia} 
\email{petar.pavesic@fmf.uni-lj.si}
\thanks{$^1$ The first author is supported by Simons Foundation collaboration grant 246221 and $^2$ the third author was partially supported by the Slovenian Research Agency grant P1-02920101.}

\dedicatory{To the Memory of Sibe Marde\v si\'c}

\begin{document}
\begin{abstract}
An inverse limit of a sequence of covering spaces over a given space $X$ is not, in general, a covering space over $X$
but is still a \emph{lifting space}, i.e. a Hurewicz fibration with unique path lifting property. Of particular interest
are inverse limits of finite coverings (resp. finite regular coverings), which yield fibrations whose fiber is homeomorphic
to the Cantor set (resp. profinite topological group). To illustrate the breadth of the theory,  we present  in this note some curious 
examples of lifting spaces that cannot be obtained as inverse limits of covering spaces.
\end{abstract}
\maketitle

\section{Introduction}

For any continuous map $p\colon L\to X$ let $p_*\colon L^I\to X^I$ be the induced map between the respective
spaces of continuous paths: $p_*\colon(\alpha\colon I\to L)\mapsto (p\circ \alpha\colon I\to X)$.

A map $p\colon L\to X$ is said to be a \emph{lifting projection} if the following diagram
$$\xymatrix{
L^I \ar[d]_{p_*} \ar[r]^{{\rm ev}_0} & L \ar[d]^p\\
X^I \ar[r]_{{\rm ev}_0} & X}
$$
is a pull-back in the category of topological spaces. In other words, there is a natural continuous one-to-one correspondence 
between paths in $L$, and  pairs $(\alpha,l)$, where $\alpha$ is a path in $X$ and $l\in L$ with $p(l)=\alpha(0)$. In fact, 
by \cite[Theorem 1.1]{Pavesic-Piccinini} a lifting projection $p\colon L\to X$ is automatically a Hurewicz fibration. Moreover, 
given a path $\alpha\colon I\to X$ and a point $l\in L$, such that $p(l)=\alpha(0)$ there exists a unique path 
$\widetilde \alpha\colon I\to L$, such that $\widetilde \alpha(0)=l$ and $p\circ\widetilde\alpha=\alpha$. 
In particular, every covering projection $p\colon \widetilde X\to X$ is a lifting projection. 

A \emph{lifting space} is a triple $(L,p,X)$ where 
$p\colon L\to X$ is a lifting projection. As usual, $X$ is the \emph{base} and $L$ is the \emph{total space}
of the fibration, but we will occasionally abuse the terminology and refer to the space $L$ 
as a lifting space over $X$. For every $x\in X$ the preimage $p^{-1}(x)\subset L$ is the \emph{fibre} of $p$ over $x$. 
If $X$ is path-connected, then all fibres are homeomorphic, so we can speak about \emph{the} fibre of $p$. 

In his classical textbook on algebraic topology E. Spanier  \cite[Chapter II]{Spanier} develops most of the theory of covering spaces 
within the general framework of lifting spaces (which he calls \emph{fibrations with unique path-lifting property}). 
As we already mentioned, covering spaces are prime examples of lifting spaces. Conversely, if the base space $X$ is sufficiently ``nice'' 
(i.e. locally path-connected and semi-locally simply connected), then the 
coverings over $X$ are exactly the lifting spaces over $X$ with a path-connected and locally path-connected total space
(\cite[Theorem 2.4.10]{Spanier}).
However, if we allow more general base spaces or total spaces that are connected but non necessarily path-connected, then a whole
new world arises. 

Indeed, many authors have studied unusual and pathological examples of covering spaces and tried to buid a suitable general theory. 
One should mention in particular Fox's theory of overlays \cite{Fox}, an interesting generalization of the concept of coverings by 
Fischer and Zastrow \cite{F-Z}, and a theory of coverings specially geared toward locally path-connected spaces by Brodsky et al. \cite{BDLM}.
See also Dydak's short note \cite{Dydak note} that was written very much in the spirit of the present article. However, our approach is 
different in that we do not attempt to simply extend the concept of a covering space to more general bases but rather we pursue the theory 
of lifting spaces in the sense of Spanier. It includes covering spaces as a special case, but it has a richer structure even when the base 
is very simple, like a circle or a cell-complex. 

Surprisingly, it turns out that despite the generality, the theory of lifting spaces has many pleasant properties that are not shared
by covering spaces. Most notably, lifting spaces are preserved by arbitrary products,
compositions and inverse limits, which is not the case for covering spaces. While a general exposition on lifting spaces can 
be found in \cite{Conner-Pavesic} and a classification of compact lifting spaces is in preparation, in this note we describe
some curious examples of lifting spaces that somehow defy classification and reveal interesting connections
with dynamical systems and profinite groups.

\section{Examples of lifting spaces}

As explained in the Introduction, in order to find lifting spaces that are not already covering spaces we must relax some 
of the usual assumptions either about the base or about the total space. We will mostly work with very simple base spaces (the circle and 
the figure eight-space) and concentrate on the intricacies of the topology of the total space. In particular, we will assume that
the total spaces are connected, but are not necessarily path-connected or locally path-connected. Only in our final example we will describe 
an interesting lifting space over the Hawaiian earring.

It is actually not difficult to construct lifting spaces that are not covering spaces as we 
have the following characterization.
\begin{proposition}(\cite[Theorem II.2.5]{Spanier})
A map $p\colon L\to X$ is a lifting space if, and only if it is a Hurewicz fibration with totally path-disconnected fibre.
\end{proposition}
Furthermore, a locally trivial projection over a paracompact space is always a Hurewicz fibration 
(see \cite[Theorem 1.6, Corollary 1.7]{Pavesic-Piccinini}). Conversely, if $p\colon L\to X$ is a lifting space, 
then the restriction of $p$ over a contractible subspace of $X$ is locally trivial. 
\begin{theorem}
Let $X$ be a locally contractible space. Then $p\colon L\to X$ is a lifting space if, and only if $p$ is locally trivial 
and the fibres of $p$ are totally path-disconnected. 
\end{theorem}
As a matter of comparison, if $X$ is locally contractible and the fibre of a lifting space $p\colon L\to X$ is discrete, then $p$ is a 
covering projection. 

\subsection{Compactified spiral}
Let us first consider lifting spaces whose base is the circle $S^1$. 
Lifting spaces whose fibre is finite are exactly the finite coverings over $S^1$ (we assume that all spaces under consideration are 
Hausdorff). There is only one infinite covering space over the circle, namely the exponential map $\exp\colon\RR\to \ZZ$ but there are
many other lifting spaces. For example, let 
$$L:=\big\{\big(\cos t,\sin t, \frac{t}{1+|t|}\big)\mid t\in\RR\big\}\cup S^1\times\{-1,1\}\subset S^1\times [-1,1],$$
and let $p\colon L\to S^1$ be the projection map. $L$ may be depicted as a compressed spiral that is ``compactified'' by one circle
at the top and one at the bottom. The projection $p$ is clearly locally trivial while the fibre is compact and totally path-disconnected
(but not discrete), therefore $p$ is a lifting projection that is not a covering projection. Nevertheless, it is very close to covering space, 
in the sense that the path-components of $L$ are covering spaces over $S^1$. Lifting spaces that split into covering space components are 
said to be \emph{decomposable}.

\subsection{Irrational slope line}
To describe our second example, choose an irrational number $a\in\RR$, define 
$$L:=\{(e^{it},e^{iat})\mid t\in\RR\}\subset S^1\times S^1,$$
and let $p\colon L\to S^1$ be the projection on the first component. $L$ is the irrational slope line that winds around the torus,
and $p$ is a lifting projection, because it is clearly locally trivial and the fibre is totally path-disconnected. The completion of $L$ is 
the entire torus, so from our viewpoint this examples belongs more to the theory of foliations. 

\subsection{Mapping torus}
Every lifting space over the circle can be obtained as a mapping torus of a suitable map. Let $F$ be any totally path-disconnected space,
and let $f\colon F\to F$ be a self-homeomorphism. Then the \emph{mapping torus} of $f$ is defined as the quotient space
$$M_f:=\frac{I\times F}{(0,x)\sim (1,f(x))},$$
and the projection to the first component induces a lifting projection $p\colon M_f\to S^1$. 
There is an obvious extension of this construction to the case when $X$ is a 1-dimensional complex (a graph). 
In fact, we may split every 1-simplex in $X$ in the middle, take the product of the resulting space with $F$, and then glue together 
the fibres over 
the cuts with respect to some choice of self-homeomorphisms of $F$. The main drawback of this 
universal construction is that it is usually very difficult to determine the properties of the resulting lifting space, e.g. whether
it is connected, path-connected, are the path-components dense, what is the fundamental group of $L$ and so on. 

\subsection{Solenoids}
From our viewpoint, the most interesting examples of lifting spaces arise as inverse limits of covering spaces (cf. 
\cite[Proposition 2.2]{Conner-Pavesic}). Let us consider the inverse system of coverings over $S^1$ obtained by iterating 
2-fold coverings:
$$\xymatrix{S^1 &  S^1 \ar[l]_2 &  S^1 \ar[l]_2 &   \cdots\ar[l]_2 }$$
The inverse limit of the sequence is the well-known \emph{dyadic solenoid} $\Sol_2$, and the resulting projection $p\colon \Sol_2\to S^1$ 
determines a lifting space whose fibre can be identified with the inverse limit of the sequence of groups 
$$\xymatrix{\ZZ/2 & & \ZZ/4 \ar[ll]_{\mod 2} & & \ZZ/8 \ar[ll]_{\mod 4} & &  \cdots\ar[ll]_{\mod 8} }$$
which is the profinite group of 2-adic integers $\widehat\ZZ_2$. Recall that $\widehat\ZZ_2$ is a compact, totally disconnected topological
group, and that as a topological space it is homeomorphic to the Cantor set. 

Lifting spaces that are inverse limits of regular coverings have many pleasant properties: every path-component of the total space is
dense in the total space, fibres are complete topological groups (even compact groups, if all coverings in the system are finite), each point 
in the fibre uniquely determines an automorphism of the lifting space, the automorphism group acts transitively on the fibres, 
it is possible to formulate a lifting theorem, ... (cf. \cite[Sections 4.,5.]{Conner-Pavesic}). It would be therefore highly desirable
to have a criterion that could recognize inverse limits of regular coverings. For instance, we are able
to tell that our first two examples (the compactified spiral and the irrational slope line) are not inverse limits of coverings 
by direct inspection of the respective fibres. 
But what if the fibres are Cantor sets, possibly provided with a specific topological group structure, and the total space is 
constructed using the universal method described above? We are going to illustrate potential difficulties with three examples,
which are also of independent interest, because of their relations with other mathematical theories. 

\subsection{The Morse-Thue lifting space over the circle}
\label{subsec:M-T}
Our next example arises from a dynamical system related to the famous \emph{Morse-Thue sequence}. It is a binary sequence (i.e. an infinite
word in the alphabet $A:=\{0,1\}$) with some outstanding properties that can be obtained as follows. Start from the word $01$ and then 
recursively replace every $0$ in the string by $01$ 
and every $1$ by $10$.  Alternatively, we may start from the string $0$ and then, 
from every $a_n$ construct $a_{n+1}$ by concatenating $a_n$ with $a_n'$ where $a_n'$ is obtained by swapping zeroes and ones in $a_n$.
Either way, the iteration of the process yields an infinite sequence whose initial portion is 
$$\mathbf{01101001~10010110~10010110~ 01101001}\ldots, $$
and which has interesting applications in dynamical systems, game theory, fractals, combinatorics and elsewhere. 

By concatenating the Morse-Thue sequence with its reverse we obtain a doubly infinite (i.e. $\ZZ$-indexed) sequence $\omega_0\in A^\Z$. 
Denote by $M_0\subset A^\Z$ the topological closure of the orbit of $\omega_0$ with respect to the shift operator $\sigma$ on $A^\Z$. 
Clearly, $M_0$ is homeomorphic to the Cantor set, and $\sigma\colon M_0\to M_0$ is a self-homeomorphism.

The \emph{Morse-Thue lifting space} is defined as the projection $p\colon L\to S^1$, where $L$ is the mapping torus of the action of $\sigma$ on $M_0$. 
Its fibres are compact and totally disconnected, and at first sight it resembles the previously described solenoid lifting space. 
However, that impression is false: the Morse-Thue lifting space cannot be obtained as an inverse limit of finite coverings of the circle
(see Theorem \ref{t:MS-fibration}).

\subsection{The 2-3 solenoid over the figure eight}
We will see later that the Morse-Thue lifting space is not an inverse limit of coverings despite its fibre is a profinite space. 
In fact, the main obstruction are the properties of the shift map on the fibre (corresponding to the parallel transport around the circle) 
that cannot result from a strict inverse limit of finite quotients of $\ZZ$. In this subsection we describe a lifting space over the wedge 
of two circles that is an inverse limit of coverings over each circle, but the transport along arbitrary paths leads to certain irregular behaviour. 

Let $\Sol_k$ denote the inverse limit of the sequence of $k$-fold coverings 
$$\xymatrix{S^1 &  S^1 \ar[l]_k &  S^1 \ar[l]_k &  \cdots\ar[l]_k },$$
so that $\Sol_k$ is a locally trivial lifting space over $S^1$ with fibre $\widehat\ZZ_k$ (i.e., a \emph{Cantor bundle}). 
In particular, we consider $\Sol_2$ and $\Sol_3$,
the 2-adic and 3-adic solenoids and the corresponding lifting projections $p\colon \Sol_2\to S^1$ and $q\colon \Sol_3\to S^1$.  
In both case the fibres are homeomorphic to the Cantor set, so we may choose a homeomorphism 
$f\colon p^{-1}(1)\to q^{-1}(1)$ and construct the amalgamated union 
$$\Sol_2\cup_f\Sol_3=\frac{\Sol_2\coprod\ \Sol_3}{x\sim f(x)}.$$
There is an obvious projection $r\colon \Sol_2\cup_f\Sol_3\to S^1\vee S^1$ from the \emph{2-3 solenoid} to the wedge of two 
circles (\emph{figure eight}).

Since $f$ is a homeomorphism, it is easy to show that restriction of $r$ to a small neighbourhood of the wedge point in $S^1\vee S^1$ 
is trivial. As it is clearly trivial over the neighbourhoods of other points, we conclude that $r$ is locally trivial, hence a fibration, 
by the Hurewicz uniformization theorem \cite[Theorem 1.6]{Pavesic-Piccinini}. 
It follows that $r$ is a lifting projection. Moreover, $r$ is indecomposable in the sense that every path-component of the total
space is dense. In fact, each path-component contains at least one path component of $\Sol_2$ and of $\Sol_3$, 
and so its closure is the entire space. 

\subsection{The squaring lifting space over the Hawaiian earring}
\label{sec:squaring}
Our last example is a lifting space over the Hawaiian earring that is obtained as an inverse limit of a sequence of two-fold coverings.
Recall that the \emph{Hawaiian earring} is the space that can be represented as a union of planar circles with shrinking radii and 
passing through the origin:
$$H_0:=\bigcup_{n=1}^\infty C_n\ ,\ \ \ \text{where}\ \ \ \ C_n:=\big\{(x,y)\in\RR^2\mid x^2+y^2=\frac{x}{n} \big\}.$$
We are going to build a sequence of coverings 
$$\xymatrix{H_0 &  H_1 \ar[l]_{p_1} & H_2 \ar[l]_{p_2}& \ldots \ar[l]_{p_3}}$$
by applying at each step the following construction. Assume that the space $X$ is obtained from a path-connected space $A$ by attaching a circle at 
each point of some finite subset  $F\subset A$. 
We can conveniently identify $X$ as a subset of the product $A\times S^1$, $X=A\times\{1\}\cup F\times S^1$.
Let $\widetilde X:=A\times\{-1,1\}\cup F\times S^1\subset A\times S^1$; then a straightforward computation shows that the squaring map on $S^1$ 
induces a two-fold covering map $A\times S^1\to A\times S^1$, which in turn restricts to a two-fold covering $\widetilde X\to X$. 

To construct $H_1$ we view $H_0$ as the result of attaching the circle $C_1$ to the union of the remaining circles. Then we let $H_1:=\widetilde H_0$,
with the projection $p_1\colon H_1\to H_0$ given by the construction described in the previous paragraph. Note that $p_1^{-1}(C_2)$ consists of two
circles, so we view $H_2$ as the result of attaching those two circles to the rest of $H_1$. Then $H_2:=\widetilde H_1$, and $p_2\colon H_2\to H_1$ is 
the corresponding two-fold covering. By iterating this procedure we obtain an inverse sequence of coverings, whose limit 
$$H_\infty:=\lim_{\longleftarrow} H_n,$$
and the natural projection $p_\infty\colon H_\infty\to H_0$ determines a lifting space over the Hawaiian earring, that we call a \emph{squaring lifting space},
because at each step the the projection is essentially induced by squaring.

At first glance the squaring lifting space over the Hawaiian earring is very similar to the 2-adic solenoid $\Sol_2$. Both lifting spaces are inverse limits
of two-fold coverings induced by squaring operation; both have fibres homeomorphic to the Cantor set, endowed with the structure of 2-adic integers compatible
with the lifting projections; one can even show that the group of deck transformations acts freely and transitively on the fibres (see \cite[Section 4]{Conner-Pavesic}).
By analogy, one would expect that, as in the case of the solenoid, $H_\infty$ is not path-connected nor locally path-connected. We were thus quite surprised to discover
that $H_\infty$ is actually path-connected and even locally path-connected (see Theorem \ref{thm:HE lifting}).

\section{Properties of lifting spaces}

In this section we state and prove the main properties of the lifting space examples described above.

\subsection{The Morse-Thue lifting space}

To examine the Morse-Thue lifting space we first recall some basic concepts from topological group actions and topological dynamics (cf. \cite{deVries}). 
Given a metric space $X$ with a right action of some topological group $T$, a pair of elements 
$(x,y) \in X\times X$ is said to be {\em proximal} if $\inf_{t\in T}d(x\cdot t,y\cdot t)=0$, otherwise it is {\em distal}. 
If every pair $(x,y)$ with $y\neq x$ is distal, then $x\in X$ is a {\em distal point}.  
A space is {\em point distal} if it contains a distal point,  and it is a {\em distal space} if it contains a distal point with dense orbit.
Furthermore, a right $T$-space $X$ is {\em minimal} if $X$ does not contain any proper $T$-invariant subset, and is 
{\em equicontinuous} if the elements of $T$ form an equicontinuous family of maps on $X$.

Every self-homeomorphism $t$ of a space $X$ determines an action of the additive group $\ZZ$ on $X$. The resulting \emph{dynamical system} 
is denoted $(X,t)$. In particular, $(M_0,\sigma)$ as described in Subsection \ref{subsec:M-T} is the so called \emph{Morse-Thue minimal system}
described in \cite[(2.27), p.~181]{deVries}, and we have the following facts
:
\begin{enumerate}
\item $\ZZ$-space $M_0$ is point distal and not distal, cf. \cite[p.309]{deVries}.
\item $\ZZ$-space $M_0$ is not equicontinuous, cf. \cite[Chapter V, (6.1) 8, p,~481]{deVries}.
\item $\omega_0$ is an almost periodic point, and therefore its orbit closure is minimal, 
  \cite[(2.29) Theorem, p.~182 and (2.31) Remarks, p.~183]{deVries}.
\item There are no periodic orbits.
\end{enumerate}

To compare the Morse-Thue system with inverse limits of finite dynamical systems, let as consider an inverse sequence of spaces $(X_n,p_{n+1,n})$ 
that is  {\em strict} in the sense that all bonding maps $p_{n+1,n}:X_{n+1}\to X_n$ are onto. Then we have the following characterization.

\begin{lemma}\label{l:invlim-equicont}
The inverse limit of a strict sequence of finite $\Z$-spaces $X_n$ is equicontinuous.
\end{lemma}
\begin{proof}
Since each $X_n$ is finite, there is a kernel of the action, 
$$K_n:=\{g\in\Z\mid (\forall x\in X_n) \ \ x\cdot g=x\}.$$
The kernels $K_n$ form a decreasing sequence and $(X_n,\Z/K_n)$ is an inverse system of finite actions. Let $(X,T)$ be its inverse limit, 
where $T$ is a procyclic group acting continuously on $X$. 
Since $T$ is compact the  Arzel\`a-Ascoli theorem implies that $T$ and its subgroup $\Z$ act equicontinuously on $X$.
\end{proof}
As a consequence, the Morse minimal system cannot be presented as the strict inverse limit of finite dynamical systems. In particular, 
a fibration, that has $M_0$ as a fibre, cannot be obtained as an inverse limit of finite coverings.

\begin{theorem}\label{t:MS-fibration}
Let $L$ be the mapping torus of the action of $\sigma$ on $M_0$. Then $p\colon L\to S^1$ is a lifting space with 
compact and totally disconnected fibres that cannot be obtained  as the
inverse limit of finite covering spaces over the circle.
\end{theorem}

\subsection{The 2-3 solenoid}

By its very construction, the 2-3 solenoid $\Sol_2\cup_f\Sol_3$ locally resembles an inverse limit of coverings, but we will show that its total space 
cannot be obtained in such a way. Indeed, if $r$ is an inverse limit of coverings, then there must exist a covering $u\colon \widetilde X\to S^1\vee S^1$ 
such that $r$ factors over $u$ as in the diagram

$$\xymatrix{
\Sol_2\cup_f\Sol_3 \ar[dd]_r \ar[dr]^s\\
& \widetilde X \ar[dl]^u\\
S^1\vee S^1}$$

As $\Sol_2$ is connected, the restriction $u|\colon s(\Sol_2)\to S^1$ is a connected covering over $S^1$ whose degree then must be a power of 2. 
Similarly, the restriction $u|\colon s(\Sol_3)\to S^1$ is a connected covering whose degree is a power of 3. As both 
degrees coincide with the degree of $u$, it follows that $u$ can only be the trivial covering of $S^1\vee S^1$. 

\begin{proposition}
$\Sol_2\cup_f\Sol_3$ cannot be mapped over $S^1\vee S^1$ to any non-trivial covering space over $S^1\vee S^1$. As a consequence, 
$r\colon \Sol_2\cup_f\Sol_3\to S^1\vee S^1$ cannot be obtained as an inverse limit of coverings of $S^1\vee S^1$.
\end{proposition}

Let us discuss another peculiar property of the 2-3 solenoid. Given a lifting space $p\colon L\to X$,
the \emph{group of deck transformations} $A(p)$ consists of all self-homeomorphisms $h\colon L\to L$ that commute with $p$, as in
$$\xymatrix{
L\ar[rr]^h \ar[dr]_p& & L\ar[dl]^p\\
& X}$$
We have already mentioned that the group of deck transformations of an inverse limit of regular coverings is in one-to-one
correspondence with the points in the fibre, just like in the case of regular coverings. And like for coverings, a small
group of deck transformations indicates that a lifting space is highly irregular. \\

\begin{theorem}
The group of deck transformations of the 2-3 solenoid is trivial.
\end{theorem}
\begin{proof}
Let $h\colon \Sol_2\cup_f\Sol_3\to \Sol_2\cup_f\Sol_3$ be a deck transformation of $r$. Clearly, the restriction $h|_{\Sol_2}$ is a deck
transformation of $p\colon\Sol_2\to S^1$. Since $p$ is regular, $h$ is completely determined by the image of one point,
say $x_0\in \Sol_2$. As a consequence, $h$ can be written as a translation
$$h(x)=x+_2(h(x_0)-_2 x_0),$$
where $+_2$ and $-_2$ denote 2-adic addition and subtraction. Similarly, $h|_{\Sol_3}$ is a translation with respect to 
the 3-adic addition. Since the two operations must coincide along the common fibre, they are related through the homeomorphism
$f\colon\widetilde\ZZ_2\to\widetilde\ZZ_3$ by the following formula
$$ h(f(x))=f(x)+_3\big(f(h(x_0))-_3f(x_0)\big).$$
If we define $a:=h(x_0)-_2x_0$ and $b:=f(h(x_0))-_3f(x_0)$ and take into account the relation $h(f(x))=f(h(x))$, 
we obtain that for every $x\in\widehat\ZZ_2$
$$f(x+_2a)=f(h(x))=h(f(x))=f(x)+_3b,$$ 
so in particular
$$b=f(a)-_3 f(0).$$
Let $\bar f(x):=f(x)-f(0)$, so that we get
$$\bar f(x+_2 a)=f(x+_2a)-_3f(a)=f(x)+_3f(a)-_3f(0)-_3 f(0)=\bar f(x)+_3\bar f(a).$$
By induction we obtain the following equality $\bar f(n\cdot_2 a)=n\cdot_3\bar f(a)$. 
However, $2^i\cdot_2 a$ is a Cauchy sequence in $\widehat\ZZ_2$ (converging to 0), while
$2^k\cdot_3\bar f(a)$ is not a Cauchy sequence in $\widehat\ZZ_3$, unless $\bar f(a)=0$, which would imply $b=0$ and $h=\mathrm{Id}$.
we conclude that the lifting space $r\colon \Sol_2\cup_f\Sol_3\to S^1$ is \emph{rigid} in the sense that $A(r)$ is trivial.
\end{proof}

\subsection{The squaring lifting space}

Intuitively, a lifting space obtained as an inverse limit of a sequence of coverings should resemble a solenoid so in particular, 
we may expect that its total space is neither path-connected nor locally path-connected. However, our intuition is wrong, as we will
explain by analysing the squaring lifting space over the Hawaiian earring. We begin with an alternative description of the spaces $H_n$
(notation is from Section \ref{sec:squaring}).

\begin{lemma}
The covering space $q_n=p_n\circ\ldots\circ p_1\colon H_n\to H_0$ can be identified as the pull-back of the $n$-th power squaring map as in the diagram
$$\xymatrix{
H_n \ar[rr] \ar[d]_{p_n}\pullbackcorner & & (S^1)^n\ar[d]^{2^n} \\
H_0 \ar@{^(->}[r]_-i & (S^1)^\N \, \ar[r]_-{\mathrm{pr_n}} & (S^1)^n}$$
where the bottom map is the composition of the inclusion $i$ of $H_0$ into the countable product of circles with the projection to the product of $n$ circles, 
and $2^n$ denotes the product of $n$ squaring maps on the circle.
\end{lemma}
\begin{proof}
By a direct computation, the total space of the pull-back of the covering projection $2^n\colon (S^1)^n\to (S^1)^n$ along $\mathrm{pr_n}\circ i$
can be decomposed as the union $\bigcup_{j=1}^\infty H_{n,j}$, where 
$$H_{n,j}=\big\{(x_i)\in (S^1)^\N\big|\, x_j\in S^1,\, \text{and if}\ i\ne j, x_i=1\, \text{for}\ i>n,\, \text{and}\ x_i=\pm 1\ \text{for}\ i\le n\big\}.$$
We use induction to prove that $H_n=\bigcup_{j=1}^\infty H_{n,j}$. Clearly, the Hawaiian earring can  be written as $H_0=\bigcup_{j=1}^\infty H_{0,j}$, so let us
assume that $H_{n-1}=\bigcup_{j=1}^\infty H_{n-1,j}$. Note that we can decompose $H_n$ as 
$$H_n=\bigcup_{j\ne n}H_{n-1,j}\cup \bigcup_{j\ne n}H'_{n-1,j}\cup H_{n-1,n},$$ 
where $H'_{n-1,j}$ is defined in the same way
as $H_{n-1,j}$ only with $x_n=-1$ instead of $x_n=1$, i.e.
$$H'_{n-1,j}=\big\{(x_i)\in (S^1)^\infty\big|\, x_j\in S^1,\, \text{and if}\ i\ne j,\hspace{50mm}$$ $$\hspace*{40mm} x_i=1\, \text{for}\ i>n,\, x_n=-1, \, \text{and}\ x_i=\pm 1\ \text{for}\ i\le n-1\big\}.$$
But that is exactly how $H_n$ was constructed from $H_{n-1}$ in Section \ref{sec:squaring}, by taking $A=\bigcup_{j\ne n}H_{n-1,j}$ and $F$ the $2^{n-1}$ points with coordinates $\pm 1$ in $A$.
\end{proof}

Ii is now easy to compute the inverse limit of the spaces $H_n$:

\begin{proposition}
The total space of the squaring lifting space is given by
$$H_\infty=\big\{(x_i)\in(S^1)^\N  \big|\ x_i=\pm 1 \,\text{for all but one index}\ i   \big\},$$ 
and the following pull-back diagram determines the lifting projection $p_\infty\colon H_\infty\to H_0$.
$$\xymatrix{
H_\infty \ar[r] \ar[d]_{p_\infty}\pullbackcorner &  (S^1)^\N\ar[d]^{2^\N} \\
H_0 \ar@{^(->}[r]_-i & (S^1)^\N }$$
\end{proposition}

We can visualize the space $H_\infty$ as a subspace of the Hilbert cube $I^\N$ obtained by doubling all edges in the 1-skeleton of $I^\N$  
(in a sense, since the Hilbert cube is not a CW-complex). Evidently, $H_\infty$ is path-connected.  It is also locally 
path-connected, because every point of $H_\infty$ admits arbitrarily small neighbourhoods of the form 
$(U\times I^\N)\cap H_\infty$ where $U\subset I^n$ and $U\cap H_\infty$ is path-connected. We have thus proved

\begin{theorem}
\label{thm:HE lifting}
The total space of the squaring lifting space over the Hawaiian earring is path-connected and locally path-connected
(in fact, it is a Peano continuum). 
\end{theorem}

The properties of the squaring lifting space are in certain sense transversal to those of the solenoids. For instance,
the lifting projection $p\colon \Sol_2\to S^1$ has a path-connected and semi-locally simply connected base, and is not 
a covering projection, because $\Sol_2$ is not path-connected nor locally path connected (compare comments in the Introduction).
On the other side, $p_\infty\colon H_\infty\to H_0$ has a path-connected and locally path-connected total space but is not 
a covering projection, because its base space is not semi-locally simply connected. Some further properties of the squaring
lifting space are collected in the following corollary. Recall that the fundamental group of the Hawaiian earring can be 
represented as the subgroup of the inverse limit of free groups consisting of the so-called \emph{legal words}, that is 
infinite words in the alphabet $a_1,a_2,a_3,\ldots$ where each letter appears finitely many times (see \cite{Cannon-Conner}
for more details). 

\begin{corollary}
\begin{enumerate}\item[]
\item The group of deck transformations of the squaring lifting space is given by $\Aut(p_\infty)\cong\ZZ_2^{\NN}$. 
\item The lifting of paths in the squaring lifting space induces the epimorphism $\partial\colon \pi_1(H_0)\to\ZZ_2^\N.$ 
\item The fundamental group of the total space $\pi_1(H_\infty)$ is isomorphic to $\Ker\partial$, which is the subgroup of 
$\pi_1(H_0)$ whose elements are legal words in which every letter appears an even number of times (i.e. represented by loops 
that go around each circle in the Hawaiian earring an even number of times).
\item Let $p\colon L\to H_0$ be any lifting space. There exists a map $f\colon H_0\to L$ such that $p_\infty=p\circ f$ if, and only 
if $\Ker{\partial}\subseteq p_\sharp(\pi_1(L))$.
\end{enumerate}
\end{corollary}
\begin{proof}
The lifting projection $p_\infty$ is the inverse limit of a sequence of regular coverings, therefore by \cite[Corollary 4.11]{Conner-Pavesic} $\Aut(p)$ is 
isomorphic to the inverse limit of the sequence of respective fibres, therefore $\Aut(p_\infty)\cong\ZZ_2^\N$. 

The lifting projection $p_\infty\colon H_\infty\to H_0$ is a fibration, so we may consider the relevant portion of the homotopy exact
sequence (cf. \cite[Theorem 4.5]{Conner-Pavesic}):
$$\xymatrix{
1 \ar[r] & \pi_1(H_\infty) \ar[r]^{\pi_1(p_\infty)} & \pi_1(H_0) \ar[r]^-\partial & \ZZ_2^{\N} \ar[r] & \pi_0(H_\infty) }$$
By exactness, $\pi_1(H_\infty)$ is isomorphic to $\Ker\partial$, which in turn clearly consists of legal words 
in $\pi_1(H_0)$ in which every letter appears an even number of times. Moreover, as $H_\infty$ is path-connected, $\pi_0(H_\infty)$ is 
trivial, hence $\partial$ is an epimorphism. Finally, since $H_\infty$ is path-connected and locally path-connected, the lifting 
criterion of \cite[Theorem II.4.5]{Spanier} implies (4).
\end{proof}

In conclusion, one should not view the examples described in this paper just as a bestiary of strange lifting spaces, but rather as 
a list of main obstacles and caveats for a viable classification theory of lifting spaces. We believe that indecomposable lifting spaces with 
profinite fibre and transitive group of deck transformations are the correct extension of the concept of (finite) covering spaces.


\begin{thebibliography}{99}
\bibitem{BDLM}
N.~Brodskiy, J.~Dydak, B.~Labuz, A.~Mitra, \emph{Covering maps  for locally  path-connected spaces}, Fund. Math. {\bf 218} (2012), 13 -– 46.
\bibitem{Cannon-Conner}
J.~W.~Cannon, G.~R.~Conner, \emph{The combinatorial structure of the Hawaiian earring group}, Topology and its Applications {\bf 106} (2000), 225 –- 271.
\bibitem{Conner-Pavesic}
G.~R.~Conner, P.~Pave\v si\'c, \emph{General Theory of Lifting Spaces}, preprint.
\bibitem{Dydak note}
J.~Dydak, \emph{Anomalous coverings}, arXiv:1210.3733
\bibitem{Fox}
R.~H.~Fox, \emph{On shape}, Fund.Math. {\bf 74} (1972), 47 –- 71. 
\bibitem{F-Z}
H.~Fischer, A.~Zastrow, \emph{Generalized  universal  coverings  and  the  shape  group}, Fund. Math. {\bf 197} (2007), 167 –- 196.
\bibitem{Pavesic-Piccinini}
P.~Pave\v si\'c, R.A.~Piccinini, \emph{Fibrations and their Classification}, Research and Exposition in Mathematics, Vol. 33 
(Heldermann Verlag, 2013)
\bibitem{Spanier} 
E.~H.~Spanier: \emph{Algebraic Topology}, (Springer-Verlag, New York 1966).
\bibitem{deVries}
J. de Vries, \emph{Elements of Topological Dynamics}, 1993
\end{thebibliography}
\end{document}